\documentclass[11pt,a4paper]{amsart}

\usepackage{amscd,amssymb,amsthm}
\usepackage{graphicx}

\numberwithin{equation}{section}

\theoremstyle{plain}
\newtheorem{theorem}{Theorem}[section]

\newtheorem{lemma}[theorem]{Lemma}
\newtheorem{proposition}[theorem]{Proposition}

\theoremstyle{definition}
\newtheorem{definition}{Definition}[section]

\theoremstyle{remark}

\newtheorem{example}{Example}[section]

\numberwithin{equation}{section}

\numberwithin{table}{section}

\numberwithin{figure}{section}

\newcommand{\bea}{\begin{eqnarray*}}
\newcommand{\eea}{\end{eqnarray*}}
\newcommand{\bean}{\begin{eqnarray}}
\newcommand{\eean}{\end{eqnarray}}

\begin{document}
\title[Intertwining operator associated to the complex Dunkl operator of type $G(m,1,N)$]{Intertwining operator associated to the complex Dunkl operator of type $G(m,1,N)$}
\author{Fethi Bouzeffour}
\address{Department of mathematics, College of Sciences\\ King Saud University,
 P. O Box 2455 Riyadh 11451, Saudi Arabia.} \email{fbouzaffour@ksu.edu.sa}
\author{Sami Ghazouani }
\address{Institut Pr\'eparatoire aux Etudes d'Ing\'enieur de Bizerte,
Universit\'e de Carthage, 7021 Jarzouna, Tunisie}
\email{Ghazouanis@yahoo.fr} 


\begin{abstract}In this work, we consider the Dunkl complex reflection operators
related to the group $G(m,1,N)$ in the
complex plane  \begin{align*}
T_i=\frac{\partial}{\partial z_i}+k_0\sum_{j\neq i}\sum_{r=0}^{m-1}\frac{1-s_i^{-r}(i,j)s_i^r}
{z_i-\varepsilon^r z_j}+\sum_{j=1}^{m-1}k_j\sum_{r=0}^{m-1}\frac{\varepsilon^{-rj}s_i^r}{z_i}, \,\,1\leq i\leq N.
\end{align*}
We first review the theory of Dunkl operators for complex reflection
groups we recall some results related to the hyper--Bessel
functions, which are solutions of a higher order differential
equation. Secondly, we construct a new explicit intertwining
operator between the operator $T_i$ and the partial derivative
operator $\frac{\partial}{\partial x_i}.$ As application we given an
explicit solution of the system: $$ T_if(x)=\kappa\lambda_i
f(x),\,\,\, f(0)=1.$$
\end{abstract}

\maketitle
\section{Introduction}
The idea of intertwining operator $V$ such that $VP = QV$ for $P$
and $Q$ ordinary differential operators goes back to Gelfand,
Levitan, Marchenko, Naimark, Delsarte and Lions (see \cite{levitan},
\cite{Lions}). It was picked up again by C. F. Dunkl, R\"{o}sler and
K. Trim\`eche (see \cite{Dunkl}, \cite{Rosler} \cite{Trimeche}) who
established some fundamental ideas related to the class of
differential difference operators. In this work we investigate in
the rank one case the particular cases of complex reflection Dunkl
operator $T(k)$, associated with complex reflection group
$G(m,1,1)$, on the set of radial rays
$U=\cup_{j=1}^m\varepsilon^j\mathbb{R}$, which is given by
\cite{Dunklref}
\begin{equation}
T(k)f(x):=\frac{df(x)}{dx}+\sum_{i=1}^{m-1}\frac{k_i}{x}\sum_{j=0}^{m-1}\varepsilon^{-ij}f(\varepsilon^jx),\,\,
\varepsilon=e^{\frac{2i\pi}{m}}\,\,\,\text{and}\,\,\,\,
 k_i\in\mathbb{C}. \label{diffh}
 \end{equation}
In particular, when $m=2$, $T(k)$ coincides with the following Dunkl
operator on the real line
\begin{equation}
Tf(x):=\frac{df(x)}{dx}+\frac{\nu+1/2}{x}(f(x)-f(-x)).\label{Dunklop}
\end{equation}
First, we indicate briefly some results involving intertwining
operators. In \cite{Dunkl}, C. F. Dunkl has proved that there exists
a linear isomorphism $V$, called the Dunkl intertwining operator,
from the space of polynomials on $\mathbb{R}$ of degree $n$ onto
itself, satisfying the transmutation relation
 \begin{equation} T\circ V=V\circ
\frac{d}{dx}\label{trans1},\,\,\,V(1)=1.
\end{equation} In \cite{Rosler}, R\"{o}sler  has obtained an integral representation of $V$  and K. Trim\`eche \cite{Trimeche} extended it to a topological isomorphism from
$\mathcal{E}(\mathbb{R}),$ the space of even $C^{\infty}$-functions
on $\mathbb{R}$, onto itself satisfying the relation \eqref{trans1}
and obtained the following form
\begin{equation}
V(f):=\mathcal{R}_{\nu}(f_e)+\frac{d}{dx}\circ\mathcal{R}_{\nu}\circ
I(f_o),
\end{equation}
where $f_e$ and $f_o$ are respectively the even and odd parts of the
function $f$, \begin{equation} I(f)(x):=\int_0^{x}f(t)dt
\label{rightinv}
\end{equation}
 and $\mathcal{R}_{\nu}$ is the
Riemann-Liouville operator given by \begin{equation}
 \mathcal{R}_{\nu}(f)(x):=\frac{\Gamma(\nu+1)}{\Gamma(\frac{1}{2})\Gamma(\nu+\frac{1}{2})}\int_0^1
(1-t)^{\nu-\frac{1}{2}}t^{-\frac{1}{2}}f(xt^{\frac{1}{2}})dt.\label{Riemann1}
\end{equation}\\ The goal of this paper is to provide a similar
construction for an intertwining operator $V_m$ between the complex
Dunkl operator $T(k)$ and the derivative operator $\frac{d}{dx}$.
Our construction is based on some hyper-Bessel operator and
Riemann-Liouville type transform.\\
The remaining sections of this paper are organized as follows. In
Section 2, we first recall notations and some results for Dunkl
operator, we establish a new representation for the complex Dunkl
operator by using circular matrices. In section 3, we discuss some
results satisfied by the hyper-Bessel functions which can be found
in the literature. In section 4, we give a new intertwining operator
between $T(k)$ and $\frac{d}{dx}$.
\section{Complex Dunkl operators of type $G(m,1,N)$} Let $m \in \mathbb{N}$
$(m\geq 2)$.  We denote by $G$ the cyclic group generated by
$\varepsilon=e^{\frac{2i\pi}{m}}$ and by
\begin{equation}U=\cup
_{j=1}^m \varepsilon^j\mathbb{R}\end{equation}a set of radial rays
in complex plane. For $i=1,\,...,\,m,$ we define the operators
\begin{equation}
p_{i}(f)(x)=\frac{1}{m}\sum_{j=0}^{m-1}\varepsilon^{-ij}f(\varepsilon^jx).
\end{equation}
These obey
\begin{equation}
id=\sum_{i=1}^{m}p_{i},\,\,\,\,\,p_{i}p_{j}=\delta_{ij}p_{i}.
\end{equation}
Then, the elements $p_{i}$ are idempotents which are generalizations
of the primitive idempotents (1 - s)/2 and (1 + s)/2 for a real
reflection $s$.
\begin{definition}A function $f:\,U \rightarrow \mathbb{C}$
is called of type $j$  with respect to  $G$, if
$$f(\varepsilon x)=\varepsilon^jf(x),$$ hold
for every $x\in U.$
\end{definition}
  \begin{lemma}Let $f$ be a function  $f:U\rightarrow
\mathbb{C}$. Then, $f$ can be decomposed  uniquely in the form
\begin{equation*}
f=\sum_{j=0}^{m-1}f_{j}\,,
\end{equation*} where the component function $f_{j}$ is of type $j$, given by
\begin{align}
f_{j}=p_{j}(f).\label{idem}
\end{align}
\end{lemma}
\begin{example}
Let $\kappa=e^{\frac{i\pi}{m}}.$ By using the previous Lemma we
obtain easily the following decomposition of the exponential
function $e^{\kappa x}$
$$e^{\kappa x}= \cos_m(x)+\sum_{l=1}^{m-1}\kappa^l \sin_{m,l}(x),$$ where
the hyper-trigonometric functions $\cos_m(x)$ and $\sin_m(x)$ are
given by \cite{Erdely}
\begin{align}
\cos_{m}(x):=\sum_{n=0}^{\infty}(-1)^n\frac{x^{nm}}{(nm)!}\,\,\,\text{and}\,\,\,\,\sin_{m,l}(x):=\sum_{n=0}^{\infty}(-1)^n\frac{x^{nm+l}}{(nm+l)!}.\label{cosine}
\end{align}
The function $y(x)=\cos_m(\lambda x)$ is the unique
$C^{\infty}$-solution of the system
\begin{equation*}\left\{
    \begin{array}{c}
     y^{(m)}(x)=-\lambda^m y(x),\\
      y(0)=1,\,y^{(1)}(0)=\,...\,=y^{(m-1)}(0)=0.
    \end{array}
  \right.
  \end{equation*}
\end{example}
We denote by $\mathcal{E}(U)$  the space of $C^{\infty}$-complex
valued functions on $U$ equipped with the topology of uniform
convergence on compacts of the functions and all their derivatives,
is a Frechet space and we denote by $\mathcal{E}_j(U)$ the subspace
of $\mathcal{E}(U)$ of functions of type $j$ with respect to the
group $G$. Of course we have
\begin{align*}
\mathcal{E}(U)=\bigoplus_{j=0}^{m-1}\mathcal{E}_{j}(U).\end{align*}
Let $\nu = (\nu_{1},\,...,\,\nu_{m-1},0)\in \mathbb{C}^{m}$ and $k=(k_1,\,...\,k_{m-1},\,0),$ with  $k_j=m\nu_j+m-j$.\\
  The complex reflection Dunkl operator associated to cyclic $G$ generated by $\varepsilon=e^{{2i\pi}{m}}$ is defined by
(\cite{Dunklref}, \cite{Bouzaffour})
 \begin{equation}
T(k)f(x):=\frac{df(x)}{dx}+\sum_{i=1}^{m-1}\frac{k_i}{x}\sum_{j=0}^{m-1}\varepsilon^{-ij}f(\varepsilon^jx).
 \end{equation}

\begin{proposition}The operator $T(k)$ can be written in the the form
\begin{align*}
T(k)=\frac{d}{dx}+\frac{\omega_k}{x},
\end{align*}
where $$\omega_k(f)=<\Omega\Lambda (f),\,k>,$$ $\Omega$ is the
Fourier $m\times m$ matrix, which is given by
$\Omega=(\varepsilon^{-(i-1)(j-1)})_{i,j}$ and $\Lambda(f)(z)$ is
the vector valued function form $U$ into $\mathbb{C}^{m}$, given by
$\Lambda f(x)=^t(f(x),\,f(\varepsilon
x),\,...,\,f(\varepsilon^{m-1}x)).$
 \end{proposition}
\begin{proof}Put
\begin{equation*}
\omega_k:=\sum_{i=1}^{m-1}k_{i}p_{i}.
\end{equation*}
A simple calculation shows that (see \cite{Martin})
\begin{equation*}
\omega_k(f)=<\Omega \Lambda(f),\,k>
\end{equation*}
and $$T(k)f=\frac{df}{dx}+\frac{\omega_k(f)}{x}.$$
\end{proof}
\begin{lemma} 1) If $f\in \mathcal{E}(U),$ then  $T(k) (f)\in \mathcal{E}(U).$\\
2) For $j=1,\,...,\,m-1,$ we have  $$p_j\circ
\frac{d}{dx}=\frac{d}{dx} \circ p_{j+1}.$$ Furthermore, if $f\in
\mathcal{E}_j(U),$ then $T(k)( f)\in \mathcal{E}_j(U).$
\end{lemma}
\begin{proof}
This follows immediately from the fact that:\\
For $i=1,\,...,\,m-1,$
\begin{align*}
p_{i}(f)(x)=\frac{1}{m}\sum_{j=0}^{m-1}\varepsilon^{-ij}f(\varepsilon^jx)
           =x\int_0^1p_{i-1}(f^{(1)})(xt)dt.
\end{align*}
\end{proof}
\section{The hyper-Bessel functions}
Let $\nu=(\nu_1,\,...,\,\nu_{m-1})\in \mathbb{R}^{m-1},$ satisfying
 $\nu_k\geq -1+\frac{k}{m},$ we denote by
\begin{align*}&|\nu|:=\nu_1+...+\nu_{m-1},\\&
\nu+\mathbf{n}:=(\nu_1+n,\,...,\,\nu_{m-1}+n)\,(n\in \mathbb{N}),\\&
\Gamma(\nu):=\Gamma(\nu_1)...\,\Gamma(\nu_{m-1}).\end{align*} The
normalized hyper-Bessel function with vector index $\nu$ is defined
by (see, \cite{Klyu}, \cite{Dularue}, \cite{Dimov2})
\begin{align}
\mathcal{J}_{\nu,m}(x):&=(\frac{x}{m})^{-|\nu|}\Gamma(\nu+\mathbf{1})J_{\nu,m}(x)\label{normalized}
=\sum_{n=0}^{\infty}\frac{(-1)^n\Gamma(\nu+\mathbf{1})}
{n!\Gamma(\nu+\mathbf{n}+\mathbf{1})}(\frac{ x}{m})^{nm
}\nonumber.\end{align} Here $J_{\nu,m}(x)$ is the hyper-Bessel
function \cite{Dularue}. The function $\mathcal{J}_{\nu,m}(\lambda
x)$ is a unique $C^{\infty}$-solution of the following problem
\cite{Klyu}
\begin{equation}
\left\{
  \begin{array}{l l}
    B_m(f)(x)=-\lambda^m f(x), \\
   f(0)=1,\,f^{(1)}(0)=\,...\, =f^{(m-1)}(0)=0.
  \end{array} \right.
\end{equation}
where the hyper-Bessel is given by \begin{equation}
B_m=\prod_{j=1}^{m-1}(\frac{d}{dx}+\frac{m\nu_j+m-j}{x})\frac{d}{dx}.
\end{equation}
 The simplest higher order hyper-Bessel operator is the
operator of $m$-fold differentiation
\begin{align*}
\frac{d^m}{dx^m}
=x^{-m}(x\frac{d}{dx})(x\frac{d}{dx}-1)...(x\frac{d}{dx}-m+1).
\end{align*}
For $m=2$ and $a_1=2\nu+1,$ $(\nu>-1/2)$ the hyper-Bessel operator
generalizes the well known second order differential operator of
Bessel $B_2$ given by where
\begin{equation}
B_2:=\frac{d^2}{dx^2}+\frac{2\nu+1}{x}\frac{d}{dx},\label{operaBes}
\end{equation}
and the corresponding normalized Bessel function is given by
\begin{equation*}
\mathcal{J}_{\nu,2}(x):=\frac{2^{\nu}\Gamma(\nu+1)}{x^{\nu}}J_{\nu}(x),
\end{equation*}
where $J_{\nu}(x)$ is the classical Bessel function (see,
\cite{Wat}). From Corollary 2 in \cite{Klyu} we obtain the following
differential recurrence relations for the normalized hyper-Bessel
functions $\mathcal{J}_{\nu,m}(x)$
\begin{align}
&\frac{d}{dx}\mathcal{J}_{\nu,m}(x)=-\frac{(\frac{x}{m})^{m-1}}{(\nu_1+1)\,...\,(\nu_{m-1}+1)}\mathcal{J}_{\nu+\mathbf{1},m}(x),\label{recurrence1}\\&
(\frac{d}{dx}+\frac{m\nu_k}{x})\mathcal{J}_{\nu,m}(x)=\frac{m\nu_k}{x}\mathcal{J}_{\nu-e_k,m}(x),\label{recurrence2}
\end{align} where $e_k, \, (1\leq k\leq m-1)$ are the standard basis of $\mathbb{R}^{m-1}$.

\section{Intertwining operator}
Let $\nu=(\nu_1,\,...,\,\nu_{m-1})\in \mathbb{C}^{m-1}$ such that
$\Re(\nu_j)>0$. We define the fractional integrals
$\mathcal{R}_{\nu,m}$ of Riemann-Liouville type for
   $f\in \mathcal{E}_m(U)$ ($\mathcal{E}_m(U)$ the subspace
of $\mathcal{E}(U)$ of functions of type $m$) by
\begin{align}
 \mathcal{R}_{\nu,m}f(x)&:=\frac{m^{3/2}\Gamma(\nu+\mathbf{1})}{(2\pi)^{(m-1)/2}}\int_0^1G_{m-1,m-1}^{m-1,0}\left(\left.
 \begin{matrix}  \nu_{1}, \nu_{2}, ..., \nu_{m-1}\\
   -\frac{1}{m},\,...,\, -\frac{m-1}{m}\end{matrix}\right \vert t\right)f(xt^{\frac{1}{m}})dt\label{Dimovstrans},
   \end{align} where
$G_{p,q}^{m,n}\left(\left.   \begin{matrix} \ a_{1}, a_{2}, ..., a_{p}\\
  \ b_1,b_{2},...,b_q\end{matrix}\right \vert z\right)
 $ is the Meijer's function (see \cite{Erdely}). This operator intertwines the hyper-Bessel operator $B_m$
and the $m$-th differential operator $\frac{d^m}{dz^m}$
\begin{equation} B_m\circ
\mathcal{R}_{\nu,m}=\mathcal{R}_{\nu,m}\circ \frac{d^m}{dz^m},
\label{intertB}
\end{equation}
and maps the hyper-cosine function $\cos_m(\lambda x)$
\eqref{cosine} of order $m\geq 2$ into a normalized hyper-Bessel
function $\mathcal{J}_{\nu,m}$ $$\mathcal{J}_{\nu,m}(\lambda
x)=\mathcal{R}_{\nu,m}(\cos_m(\lambda\,.))(x).$$For $m=2,$
$\mathcal{R}_{\nu,m}$ is reduced to the so called Riemann-Liouville
transform $\mathcal{R}_{k}$ defined in \eqref{Riemann1}. The
operator $\mathcal{R}_{\nu}^m$ can be written also as a product of
the Erd\'elyi-Kober integrals
\begin{equation}
\mathcal{R}_{\nu,m}f(x)=\frac{m^{3/2}\Gamma(\nu+\mathbf{1})}{(2\pi)^{(m-1)/2}}\prod_{k=1}^mI_{m-1}^{(\frac{k}{m},\,\nu_k+1-\frac{k}{m})}f(x),
\end{equation}
where the Erd\'elyi-Kober fractional integrals is defined by
\begin{equation}
I^{\alpha,\,\beta}_{\gamma}f(x):=\int_{0}^1\frac{(1-t)^{\alpha-1}t^{\beta}}{\Gamma(\alpha)}
f(xt^{\frac{1}{\gamma}})dt,\,\,\,
Re(\alpha)>0,\,\,Re(\beta)>0,\,\,Re(\gamma)>0.
\end{equation} By Theorem 3.5.7 in \cite{kiryakova2}  and by similar argument as \cite{Trimeche}, we can
show that the operator $\mathcal{R}_{\nu,m}$ is a topological
isomorphism from $\mathcal{E}_m(U)$ onto itself and its inverse is
given by \begin{equation}
\mathcal{R}_{\nu,m}^{-1}f(x)=\frac{(2\pi)^{(m-1)/2}}{m^{3/2}\Gamma(\nu+\mathbf{1})}\prod_{k=1}^m
\prod_{j=1}^{n_k}(-1+j+\frac{k}{m}+\frac{1}{m}x\frac{d}{dx})I_{m-1}^{(\nu_k,\,n_k-\nu_k+\frac{k}{m}+1)}f(x),
\end{equation}where \begin{equation}
 n_k=\left\{
  \begin{array}{l l}
    [\nu_k-\frac{k}{m}+1]+1,\,\,\,\text{if}\,\,\,\, \nu_k-\frac{k}{m}\ \text{is non integer}, \\
   \nu_k-\frac{k}{m}+1,\,\,\,\text{if}\,\,\,\,\, \nu_k-\frac{k}{m}\ \text{is integer}.
  \end{array} \right.  \label{SSys1}
\end{equation}Let consider the operator $V_m$ defined for $f\in
\mathcal{E}(U)$ by
\begin{align}V_m(f)&=\sum_{j=1}^mA_{j}\circ \mathcal{R}_{\nu,m}\circ
I^{m-j}\circ p_j(f),
\end{align} where the operator $I$ is defined in\eqref{rightinv} and
\begin{align} &  A_{m}=1, \,\,\,\,A_{m-1}=\frac{d}{dx}, \,\,\,\,
A_j=\prod_{k=j+1}^{m-1}(\frac{d}{dx}+\frac{m\nu_{k}+m-k}{x})\frac{d}{dx},\,\,\,
1\leq j\leq m-2.\end{align} The operator $V_m$ is well defined on
the space $\mathcal{E}(U),$ since for $f\in\mathcal{E}(U),$ we have
 $$I^{m-j}\circ p_j(f) \in \mathcal{E}_m(U).$$
\begin{theorem}
The operator $V_m$ satisfy the following intertwining relation on
the space $\mathcal{E}(U)$
$$T(k)\circ V_m=V_m \circ \frac{d}{dx}.$$
\end{theorem}
\begin{proof} Let $f \in \mathcal{E}(U).$ It is clearly that for $j=1,\,...,\,m$, the function $$A_j\circ I^{m-j}\circ p_j(f)\in
\mathcal{E}_j(U).$$ Then,
\begin{align*}T(k)\circ V_m(f)&=\frac{d}{dz}\circ \mathcal{R}_{\nu,m}\circ
p_m(f) + \sum_{j=1}^{m-1}(\frac{d}{dx}+\frac{k_j}{x})\circ A_{j}\circ
\mathcal{R}_{\nu,m}\circ I^{m-j}\circ p_j(f)\\&=\frac{d}{dx}\circ
\mathcal{R}_{\nu,m}\circ p_m(f) + B_{m}\circ \mathcal{R}_{\nu,m}\circ
I^{m-1}\circ p_1(f)+ \sum_{j=2}^{m-1} A_{j-1}\circ
\mathcal{R}_{\nu,m}\circ I^{m-j}\circ p_j(f).\end{align*} On the other
hand from \eqref{intertB}, we can write
\begin{align*}
B_{m}\circ \mathcal{R}_{\nu,m}\circ
I^{m-1}\circ p_1=\mathcal{R}_{\nu,m}\circ
\frac{d^m}{dz^m}\circ I^{m-1}p_1=
\mathcal{R}_{\nu,m}\circ
\frac{d}{dx}\circ p_1=A_m\circ\mathcal{R}_{\nu,m}\circ p_m
\circ\frac{d}{dx}.
\end{align*}
Similarly \begin{align*} \frac{d}{dx}\circ \mathcal{R}_{\nu,m}\circ
p_m=A_{m-1}\circ\mathcal{R}_{\nu,m}\circ I\circ p_{m-1}\circ
\frac{d}{dx}.\end{align*} So that
\begin{align*}
\sum_{j=2}^{m-1} A_{j-1}\circ
\mathcal{R}_{\nu,m}\circ I^{m-j}\circ p_j&=\sum_{j=1}^{m} A_{j}\circ
\mathcal{R}_{\nu,m}\circ I^{m-j-1}\circ p_{j+1},\\&=
\sum_{j=1}^{m-2} A_{j}\circ
\mathcal{R}_{\nu,m}\circ I^{m-j}\circ\frac{d}{dx}\circ p_{j+1}\\&=
\sum_{j=1}^{m-2} A_{j}\circ
\mathcal{R}_{\nu,m}\circ I^{m-j}\circ p_j\circ \frac{d}{dx}.
\end{align*}
Thus,
\begin{align*}T(k)\circ V_m(f)=\sum_{j=1}^{m} A_{j}\circ
\mathcal{R}_{\nu,m}\circ I^{m-j}\circ p_j\circ \frac{d}{dx}(f)=V_m\circ \frac{d}{dx}(f).
\end{align*}
\end{proof}

\begin{theorem} Under the condition \begin{equation}k_j=m\nu_j+m-j\geq 0,\,\,j=1,\,...,\,m-1.\label{const}\end{equation}The following  system
\begin{equation}
 \left\{
  \begin{array}{l l}
    T(k)f(x)=\kappa\lambda f(x), \\
   f(0)=1 .
  \end{array} \right.  \label{SSys1}
\end{equation} has the following solution
\begin{align}
\mathcal{D}_{\nu}(\lambda,\, x)=\mathcal{J}_{\nu}( \lambda
x)+\sum_{j=1}^{m-1}\frac{(\kappa\lambda)^{j}}{m^j(\nu_1+1)\,...\,(\nu_{m-j}+1)}
\mathcal{J}_{(\nu_1+1,\dots,\nu_{j}+1,\nu_{j+1},...,\nu_{m-1})}(\lambda
x).\label{Dunklkernel}
\end{align}

 \end{theorem}
\begin{proof}According to Theorem 4.1, $V_m$ intertwines
 $B_m$ and $\frac{d}{dx}$ in $\mathcal{E}(\mathbb{R})$. We apply the
intertwines operator $V_m$ to the initial value problem
\begin{equation}
 \left\{
  \begin{array}{l l}
  f^{'}(x)= \kappa\lambda f(x), \\
   f(0)=1.
  \end{array} \right.  \label{SSys2}
  \end{equation}
Then if $f$ is a solution of \eqref{SSys2} then $V_m(f)$ is a solution
\eqref{SSys1}. Therefore
$\mathcal{D}(\lambda,x)=V_m(e^{\kappa\lambda\,.})(x)$ is a solution
of the system \eqref{SSys1}. Using \eqref{recurrence1} and
\eqref{recurrence2} we can write $\mathcal{D}(\lambda,x)$ in form
\eqref{Dunklkernel}.\end{proof}

\end{document}